\documentclass{article} 
\usepackage{amssymb,amsmath,amsthm}
\usepackage{amscd}
\usepackage{hyperref}
\usepackage{epsfig}
\usepackage{amsfonts}
\usepackage{amssymb}
\usepackage{amsmath,enumerate}
\usepackage{commath}
\usepackage{euscript}
\usepackage{enumitem}
\usepackage{amscd}
\usepackage{xcolor}
\usepackage[ruled,vlined]{algorithm2e}
\usepackage[numbers,sort&compress]{natbib}

\newtheorem{theorem}{Theorem}[section]
\newtheorem{remark}{Remark}[section]
\newtheorem{corollary}{Corollary}[section]
\newtheorem{definition}{Definition}[section]

\newtheorem{lemma}{Lemma}[section]
\newtheorem{proposition}{Proposition}[section]
\newtheorem{example}{Example}[section]

\begin{document}
	\title{A Class of Quasi-Variational Inequalities with Unbounded Constraint Maps: Existence Results and Applications}
	\author{
		Asrifa Sultana\footnotemark[1] \footnotemark[2] , Shivani Valecha\footnotemark[2] }
	\date{ }
	\maketitle
	\def\thefootnote{\fnsymbol{footnote}}
	
	\footnotetext[1]{ Corresponding author. e-mail- {\tt asrifa@iitbhilai.ac.in}}
	\noindent
	\footnotetext[2]{Department of Mathematics, Indian Institute of Technology Bhilai, Raipur - 492015, India.
	}
	
	
		
	\begin{abstract}
	The quasi-variational inequalities play a significant role in analyzing a wide range of real-world problems. However, these problems are more complicated to solve than variational inequalities as the constraint set is based on the current point. We study a class of quasi-variational inequality problems whose specific structure is beneficial in finding some of its solutions by solving a corresponding variational inequality problem. Based on the classical existence theorem for variational inequalities, our main results ensure the occurrence of solutions for the aforementioned class of quasi-variational inequalities in which the associated constraint maps are (possibly) unbounded. We employ a coercivity condition which plays a crucial role in obtaining these results. Finally, we apply our existence results to ensure the occurrence of equilibrium for the pure exchange economic problems and the jointly convex generalized Nash games.\color{black}
\end{abstract}


\section{Introduction}

The theory of variational inequalities, originated by Stampacchia \cite{stampacchia}, provides us with a consolidated structure to analyze a wide range of unrelated problems emerging in economics, mechanics and applied mathematics. The variational inequality (VI) where the constraint set relies on the current point is known as quasi-variational inequality (QVI). Assume that $T:\mathbb{R}^n\rightrightarrows \mathbb{R}^n$ and $K:D\rightrightarrows D$ are multi-valued maps where $D\subseteq \mathbb{R}^n$ is non-empty. Then the QVI problem denoted by $QVI(T,K)$ is to determine an element $\bar x\in K(\bar x)$ so that:
\begin{equation}\label{QVI}
	\exists\,\bar x^*\in T(\bar x) ~\text{satisfying}~\langle\bar x^*,z-\bar x\rangle\geq 0,\enspace \text{for all}~z\in K(\bar x).
\end{equation}
In particular, the problem $QVI(T,K)$ becomes a variational inequality problem $VI(T,D)$ if we consider that $K$ is a constant map. 

The aforementioned concept of QVI problems was initiated by Chan-Pang \cite{chan}. Since then, these problems have been studied intensively due to their applications in quasi-optimization problems \cite{ausselQVI}, generalized Nash equilibrium problems \cite{ausselstra, facc}, traffic network problems \cite{ausselQVI} and economic equilibrium problems \cite{donato, milasi, wets}. Some remarkable results that ensure the occurrence of solutions for $QVI(T,K)$ include a preliminary result by Tan \cite{tan}, which is proved under the upper semi-continuity assumption on the map $T$ and a recent result by Aussel-Cotrina \cite{aussel-cotr}, which is proved under the generalized monotonicity and weaker continuity assumptions on $T$. Recently, Aussel-Sultana \cite{ausselcoer} extended the results in \cite{aussel-cotr} to the case where the constraint map $K$ (possibly) admits unbounded values. It is essential to emphasize that QVI problems are more complicated to solve computationally than variational inequalities since the constraint set is based on the current point. This fact motivated Aussel-Sagratella \cite{ausselstra} to define a special class of QVI problems having reproducible constraint maps, for which the entire set of solutions can be determined by solving associated variational inequalities.

Furthermore, Aussel et. al. \cite{ausselradn} identified another class of QVI problems whose specific structure is beneficial in finding some of its solutions by solving a corresponding VI problem. In particular, the following class of QVI, denoted by $QVI(\ref{eqA})$ is considered by them: for a given function $g:\mathbb{R}^n\rightarrow \mathbb{R}^m$ and the multi-valued maps $G:\mathbb{R}^n\rightrightarrows \mathbb{R}^n$ and $M:P\rightrightarrows \mathbb{R}^n$ where $P\neq \emptyset$ is a set in $\mathbb{R}^m$, the problem $QVI(\ref{eqA})$ corresponds to,
\begin{align}
	&~\text{find}~ (\bar p,\bar y)\in P\times M(\bar p)~\text{such that}~ \exists\, \bar y^*\in G(\bar y)~\text{satisfying}\nonumber\\
	\tag{A}&\langle g(\bar y),p-\bar p\rangle + \langle \bar y^*, z-\bar y \rangle \geq 0,\enspace\text{for any}~(p,z)\in P\times M(\bar p).\label{eqA}
\end{align} 
It is noticeable that these problems were earlier considered by Donato-Milasi-Vitanza \cite{donato, milasi} in connection to the economic equilibrium problems. The primary motivation in \cite{ausselradn} for studying these problems is to check the occurrence of Radner equilibrium for a sequential trading model with uncertainty. The authors in \cite{ausselradn} proved the existence results for $QVI(\ref{eqA})$ under the weak continuity and generalized monotonicity assumptions on the map $G$, by considering that the constraint map $M$ admits bounded values. However, their assumption over $M$ is too restrictive for certain economic applications, for e.g., the feasible strategy sets of the players in economic games are not always bounded, as illustrated in Section \ref{applications}. 

The objective of this article is to check the solvability of the specific class of quasi-variational inequalities $QVI(\ref{eqA})$, in which the constraint map $M$ (possibly) admits unbounded values. On the account of specific structure of the considered quasi-variational inequalities, we establish our main existence results for such problems $QVI(\ref{eqA})$ having non-compact valued constraint maps, by employing a preliminary existence result on variational inequalities. In this case, the existence results are proved under the weak continuity and generalized monotonicity assumptions on the map $G$. Additionally, considering the fact \textcolor{black}{that the} given problem $QVI(\ref{eqA})$ is a particular instance \textcolor{black}{of classical QVI} (as illustrated in Section \ref{alternative}), an alternate existence result is derived for the considered problems. In particular, the results derived in this paper extend the existence theorems due to Aussel et al. \cite{ausselradn} to the case where constraint map admits unbounded values. Finally, we show the existence results derived by us find their applicability in establishing the existence of equilibrium for the pure exchange economic problems \color{black}in which the consumption sets are not necessarily bounded below \cite{donato, milasi,page} \color{black} and the jointly convex generalized Nash games \cite{facc}.

This article is arranged in the following sequence: Section \ref{preliminaries} consists of some key results and supporting definitions. In the Subsection \ref{coercivity}, we have first defined a coercivity condition in context to the considered problem $QVI(\ref{eqA})$ and by employing this coercivity condition, we establish various existence theorems for $QVI(\ref{eqA})$ in Subsection \ref{mainresult} and Subsection \ref{alternative}. Finally, Section \ref{applications} illustrates two applications of the theorems established by us: First, we derive the occurrence of equilibrium for a pure exchange economy in Subsection \ref{application1} and second, we derive the occurrence of equilibrium for a jointly convex generalized Nash game in Subsection \ref{application2}.

\section{Preliminaries}\label{preliminaries}
The definitions and some basic results discussed in this section will be useful for deriving the existence results in upcoming sections. For any set $A\subset \mathbb{R}^n$, we denote the interior of $A$ by $int A$ and the convex hull of $A$ by $co(A)$. Furthermore, $T:\mathbb{R}^n\rightrightarrows \mathbb{R}^m$ represents a multi-valued map, that is, $T(x)$ is some set in $\mathbb{R}^m$ for any given $x\in \mathbb{R}^n$. For the definitions of upper and lower semi-continuity of map $T$, the reader may refer \cite{aubin, kurt}. 

We now recall the concept of upper sign-continuity \cite{hadji} which is weaker when compared to upper semi-continuity criterion. For a given convex subset $K$ of $\mathbb{R}^n$, a map $T:\mathbb{R}^n\rightrightarrows \mathbb{R}^n$ is known to be upper sign-continuous over $K$ if for every $u,v\in K$ and $u_s=su+(1-s)v$ we have, 
$$\text {for all}~s\in(0,1),~ \inf_{u_s^*\in T(u_s)}\langle u_s^*, v-u\rangle\geq 0~\Rightarrow~ \sup_{u^*\in T(u)}\langle u^*, v-u\rangle \geq 0.$$
The notion of upper sign-continuity, initiated in \cite{hadji}, was further refined in \cite{aussel-hadj} to local upper sign-continuity. A map $T$ is known to be \textit{locally upper sign-continuous} at some $u\in \mathbb{R}^n$ if there is some convex open set $\mathcal O_u$ containing $u$ and there exists a map $\psi_u:\mathcal O_u\rightrightarrows \mathbb{R}^n$ such that $\psi_u$ meets upper sign-continuity property over $\mathcal O_u$ and $\psi_u(\tilde u)$ is non-empty convex compact subset of $T(\tilde u)\setminus \{0\}$ for any $\tilde u\in \mathcal O_u$. 

A map $T$ is known to be dually lower semi-continuous over $K\subseteq\mathbb{R}^n$ if for each $v\in K$ and any sequence $\{u_n\}_{n\in \mathbb{N}}$ in $K$ with $u_n\rightarrow u$, we have,
$$ \liminf_n \sup_{u_n^*\in T(u_n)} \langle u_n^*, v-u_n\rangle \leq 0 \Rightarrow \sup_{u^*\in T(u)} \langle u^*,v-u\rangle\leq 0.$$

Following definitions are concerned with various types of generalized monotonicity \cite{aussel-cotr} for a multi-valued map $T:\mathbb{R}^n\rightrightarrows \mathbb{R}^n$. It is considered to be:
\begin{itemize}
	\item[-] pseudomonotone over $K\subseteq \mathbb{R}^n$ if for every $v,w\in K$ we have,
	$$ \exists\,v^*\in T(v)~\text{s.t.}~\langle v^*,w-v\rangle \geq 0\Rightarrow \langle w^*,w-v\rangle \geq 0,\enspace \text{for all}~ w^*\in T(w);$$
	\item[-] quasimonotone over $K\subseteq \mathbb{R}^n$ if for every $v,w\in K$ we have,
	$$ \exists\,v^*\in T(v)~\text{s.t.}~\langle v^*,w-v\rangle > 0\Rightarrow \langle w^*,w-v\rangle \geq 0,\enspace \text{for all}~ w^*\in T(w).$$ 
\end{itemize}
It is obvious that any pseudomonotone map fulfils quasimonotonicity property also. But the converse can be disproved by considering an example of the map $T:\mathbb{R}\rightrightarrows\mathbb{R}$ defined as $T(v)=\{v^2\}$. Clearly, $T$ is quasimonotone on the set $[-1,1]$ but not pseudomonotone.

We now recall some preliminary concepts related to solution set for variational inequality problems \cite{aussel-cotr}. Assume that $T:\mathbb{R}^n \rightrightarrows \mathbb{R}^n$ is a multi-valued map. For a given non-empty set $K \subset \mathbb{R}^n$, we use $VI(T,K)$ and $VI^*(T,K)$, respectively, to denote the solution set and star solution (or non-trivial solution) set of Stampacchia VI problem. These sets are given as follows:
\begin{align*}
	&VI(T,K)=\{\tilde u\in K|\,\exists\, \tilde u^*\in T(\tilde u)~\text{satisfying}~\langle\tilde u^*,v-\tilde u\rangle \geq 0,\,\forall\, v\in K\},\\
	VI&^*(T,K)=\{\tilde u\in K|\,\exists\, \tilde u^*\in T(\tilde u)\setminus \{0\}~\text{satisfying}~\langle\tilde u^*,v-\tilde u\rangle \geq 0,\,\forall\, v\in K\}.
\end{align*}
Further, we use $MVI(T,K)$ to denote the solution set of Minty VI problem, which is given as follows:
$$MVI(T,K)=\{\tilde u\in K|\,\langle v^*,v-\tilde u\rangle\geq 0,\,\forall\, v\in K,\forall\, v^*\in T(v)\}.$$
One can easily verify that $MVI(T,K)$ is a closed convex set if $K\subseteq\mathbb{R}^n$ is closed and convex. Furthermore, $VI(T,K)\subseteq MVI(T,K)$ if the map $T$ fulfils pseudomonotonicity property.
Following result which ensures the non-emptiness of solution sets is obtained by combining \cite[Lemma 2.1]{aussel-hadj} and \cite[Proposition 2.1]{aussel-hadj}:
\begin{lemma}\cite{aussel-hadj}\label{lemmaNE}
	Consider a map $T:\mathbb{R}^n\rightrightarrows \mathbb{R}^n$ and $K\neq\emptyset$ is convex compact subset of $\mathbb{R}^n$. Then,
	\begin{itemize}
		\item[(a)] $MVI(T,K)\neq \emptyset$ if $T$ is pseudomonotone over $K$;
		\item[(b)] $VI^*(T,K)\neq\emptyset$ if $T$ is quasimonotone and locally upper sign-continuous over $K$. 
	\end{itemize}
\end{lemma}

Following preliminary result which ensures the occurrence of solutions for variational inequality problems will be helpful in deriving our upcoming existence results:

\begin{theorem}\cite{saigal}\label{theoremsaigal}
	Assume that $K\neq \emptyset$ is a convex compact subset of $\mathbb{R}^n$. Suppose the map $T:K\rightrightarrows \mathbb{R}^n$ meets upper semi-continuity property and admits non-empty convex compact values. Then the corresponding VI problem consists of at least one solution.
\end{theorem}

Finally, we recall the concept of star quasi-variational inequalities \cite{ausselradn}: for a given function $g:\mathbb{R}^n\rightarrow \mathbb{R}^m$ and the multi-valued mappings $G:\mathbb{R}^n\rightrightarrows \mathbb{R}^n$ and $M:P\rightrightarrows \mathbb{R}^n$ where $P\neq \emptyset$ is a subset of $\mathbb{R}^m$, the problem $QVI^*(\ref{eqB*})$ corresponds to: 
\begin{align}
	\text{find}&~ (\bar p,\bar y)\in P\times M(\bar p)~\text{such that}~ \exists\, \bar y^*\in G(\bar y)\setminus \{0\}~\text{satisfying}\nonumber\\
	\tag{B}\langle g&(\bar y),p-\bar p\rangle + \langle \bar y^*, z-\bar y \rangle \geq 0,~\text{for any}~(p,z)\in P\times M(\bar p).\label{eqB*}
\end{align} 
Clearly, any solution of the problem $QVI^*(\ref{eqB*})$ is a non-trivial solution of the considered problem $QVI(\ref{eqA})$. Later on, we will see that the problem $QVI^*(\ref{eqB*})$ is actually motivated by an application illustrated in Subsection \ref{application1} (existence of equilibrium for a pure exchange economy with quasi-concave utility functions). In fact, the primary reason for eliminating the zero vector from the map $G$ is to ignore the trivial solution of the quasi-variational inequality problem $QVI(\ref{appl})$, which characterizes the Walrasian equilibrium problem. 
\section{Existence Results} 
\subsection{Coercivity Condition for Quasi-Variational Inequalities} \label{coercivity} 
In this subsection, we first define a coercivity condition for the considered class of quasi-variational inequality problems $QVI(\ref{eqA})$, which will be further used by us to show the occurrence of solutions for these problems. 

In connection with variational inequalities defined over unbounded constraint sets, a coercivity condition is usually assumed to ensure the occurrence of solutions for such VI problems. In \cite{aussel-hadj}, Aussel-Hadjisavvas demonstrated the occurrence of solutions for the classical variational inequality $VI(T,K)$ by employing the below mentioned coercivity criterion $(\mathcal{C})$ on the map $T$: 
\begin{align*} (\mathcal{C})\qquad \exists&\, r>0~\text{such that}~ \text{for any}~ y\in K\setminus \bar B(0,r),~\exists\, z\in K
	~\text{with}~\\ &\norm{z}<\norm{y}\text{satisfying}~\langle y^*, y-z\rangle \geq 0,~\text{for each}~y^*\in T(y).
\end{align*} 
Furthermore, Bianchi et al. \cite{bianchi} presented and \textcolor{black}{compared various} coercivity conditions for the variational inequalities defined by generalized monotone maps.

Now, we define the coercivity condition $\mathcal{C}(p)$ for the considered problem $QVI(\ref{eqA})$ by adapting the aforementioned coercivity condition $(\mathcal{C})$. Consequently, we employ the coercivity condition $\mathcal{C}(p)$ to ensure the occurrence of solutions for $QVI(\ref{eqA})$ having unbounded constraint maps. 
Assume that $G:\mathbb{R}^n\rightrightarrows \mathbb{R}^n$ and $M:P\rightrightarrows \mathbb{R}^n$ are multi-valued mappings associated to $QVI(\ref{eqA})$. For any point $p\in P$, the below mentioned coercivity condition $\mathcal{C}(p)$ is fulfilled at $p$ if:
\begin{align*}
	\mathcal{C}(p):\qquad&\exists\, r_p>0~\text{such that}~\forall\, y\in M(p)\setminus \bar B(0,r_p), \exists\, z\in M(p)~\text{with} \\&~\norm z<\norm y~\text{satisfying}~ \langle y^*,y-z\rangle\geq 0,~\text{for each}~y^*\in G(y).
\end{align*}
Note that the coercivity condition $\mathcal{C}(p)$ reduces to the condition $(\mathcal{C})$ if the multi-valued map $M$ is a \textcolor{black}{constant} map. Additionally, if $M(p)$ is bounded at some $p\in P$, then it is simple to verify that the condition $\mathcal{C}(p)$ is fulfilled at that point $p$. In fact, in this case, one may just assume that $r_p >\max \{\norm{y}|\,y\in M(p)\}$.

Following example reveals that the coercivity condition $\mathcal{C}(p)$ is easily applicable to the quasi-variational inequalities having unbounded constraint maps:

\begin{example} \label{eg1}
	Suppose the map $G:\mathbb{R}^2\rightrightarrows \mathbb{R}^2$ is defined as,
	\begin{align*}
		G(x_1,x_2)=\begin{cases}
			\{(2,2)\}~&\text{if}~x_2\neq 0\\
			\{(s,s)~\text{s.t.}~1\leq s\leq 2\}~&\text{if}~x_2=0.
		\end{cases}
	\end{align*}
	For the given set $P=[0,1]\times [0,1]$ and some fixed $\eta\in (0,1)$, the map $M:P\rightrightarrows \mathbb{R}_+^2$ is defined as,
	\begin{align*}
		M(p_1,p_2)=\begin{cases}[0,\eta]\times [0,\infty)~&\text{if}~0\leq p_1\leq \eta\\ [0,p_1]\times [0,\infty)~&\text{if}~\eta< p_1\leq 1.
		\end{cases}
	\end{align*}
	Then, one can verify that the coercivity condition $\mathcal{C}(p)$ holds for any $p\in P$ by considering $r_p=1$.
\end{example}
	\subsection{Variational Inequality Reformulation of Quasi-Variational Inequalities and General Existence Results}\label{mainresult}
	
	In this subsection, our primary objective is to give sufficient conditions which guarantee the presence of a solution for the considered problem $QVI(\ref{eqA})$ defined by an unbounded constraint map. In this regard, we utilize the specific structure of the problem $QVI(\ref{eqA})$, which is beneficial in finding some of its solution by solving a corresponding variational inequality.
	
	The below mentioned result is our first main theorem which ensures the occurrence of solutions for the problem $QVI(\ref{eqA})$ having non-compact valued constraint map $M$. In particular, this result is established under the weak continuity and pseudomonotonicity assumption on the map $G$, by employing a preliminary existence result for VI and the coercivity condition $\mathcal{C}(p)$. 
		\begin{theorem}\label{theorempseudo}
			Assume that $g:\mathbb{R}^n\rightarrow \mathbb{R}^m$ is an affine map and $P\subset \mathbb{R}^m$ is non-empty, convex and compact. Suppose that,
			\begin{enumerate}[label=(\roman*),ref=(\roman*)]
				\item \label{i} $M:P\rightrightarrows \mathbb{R}^n$ is convex valued, lower semi-continuous and closed mapping where the interior of the set $M(p)$ is non-empty for each $p \in P$;
				\item $G:\mathbb{R}^n\rightrightarrows \mathbb{R}^n$ satisfies the local upper sign-continuity and pseudomonotonicity over the set $co(M(P))$.
			\end{enumerate}
			Then $QVI(\ref{eqA})$ consists a solution if the coercivity condition $\mathcal C(p)$ holds for each $p\in P$ and there exists some $r>sup\{r_p\,|\,p\in P\}$ such that $M(p)\cap \bar B(0,r)\neq \emptyset$ for any $p\in P$. 
		\end{theorem}
		\color{black}
		\begin{proof}
			Define a multi-valued map $M_r: P\rightrightarrows \mathbb{R}^n$ as $ M_r(p)= M(p) \cap \bar B(0,r).$ Suppose $\Phi_r:P\rightrightarrows \mathbb{R}^n$ is defined as $\Phi_r(p)= VI(G,M_r(p))$, that is,
			\begin{align}
				\Phi_r(p)=\{\bar y\in M_r(p)|\,\exists\,\bar y^*\in G(\bar y)~\text{satisfying}~\langle \bar y^*,y-\bar y \rangle\geq 0,~\forall y\in M_r(p) \}
			\end{align}
			and $\Gamma_r:P\rightrightarrows \mathbb{R}^m$ is defined as $\Gamma_r(p)=g\circ \Phi_r(p)$. We first aim to show that $VI(\Gamma_r,P)$ admits a solution. We observe that for any $\bar p\in P$ solving $VI(\Gamma_r, P)$, there exists some $\bar p^*\in \Gamma_r(\bar p)=g\circ\Phi_r(\bar p)$ or equivalently $\bar p^*=g(\bar y)$ for some $\bar y\in \Phi_r(\bar p)$. Hence, $\bar p\in P$ solving $VI(\Gamma_r,P)$ and $\bar y\in M(\bar p)\cap \bar B(0,r)$ solving $VI(G,M_r(\bar p))$ satisfy the following,
			\begin{align}
				\langle g(\bar y), p-\bar p\rangle &\geq 0,\enspace\text{for all}~ p\in P,~\text{and}\label{eq1}\\
				\exists\, \bar y^*\in G(\bar y), \langle \bar y^*, z-\bar y\rangle &\geq 0,\enspace \text{for all}~ z\in M(\bar p)\cap \bar B(0,r).\label{eq2}
			\end{align}
			
			Evidently, the map $M_r$ is closed which admits convex and compact values due to hypothesis \ref{i}. We claim that $int M_r(p)\neq \emptyset$ for any $p\in P$. Since $M(p)\cap \bar B(0,r)\neq \emptyset$ and the coercivity condition $\mathcal C(p)$ holds at $p$, we can easily show the existence of some point in $M(p)\cap B(0,r)$. This fact, along with the hypothesis that the set $M(p)$ is convex with non-empty interior, leads us to the conclusion $int M_r(p)\neq \emptyset$. Eventually, the map $M_r$ becomes lower semi-continuous according to \cite[Lemma 1]{ausselcoer}. In brief, $M_r$ is closed and lower semi-continuous map where the set $M_r(p)$ is convex compact having non-empty interior for each $p\in P$.
			
			Now, the maps $G$ and $M_r$ fulfil all the conditions in \cite[Proposition 4.2]{aussel-cotr}, which proves $\Phi_r$ is a closed map. Hence, the map $\Phi_r$ meets the upper semi-continuity property due to the fact $\Phi_r(P)\subseteq M(P)\cap \bar B(0,r)$ is a compact set. \color{black}As per \cite[Lemma 3.1(i)]{aussel-cotr}, the local upper sign-continuity of $G$ yields $VI(G,M_r(p))\supseteq MVI(G,M_r(p))$. The pseudomonotonicity of the map $G$ implies $VI(G,M_r(p))\subseteq MVI(G,M_r(p))$.  \color{black} Hence, 
			$$\Phi_r(p)=VI(G,M_r(p))= MVI(G,M_r(p)),\enspace\text{for each}~p\in P.$$
			Hence, it can be concluded from Lemma \ref{lemmaNE} that $\Phi_r$ is upper semi-continuous map with non-empty, convex and compact values as the set $M_r(p)$ is compact. 
			
			This in turn implies that the map $\Gamma_r=g\circ \Phi_r:P\rightrightarrows \mathbb{R}^m$ is upper semi-continuous due to continuity of the function $g$. Furthermore, the fact that $g$ is an affine function leads us to the conclusion that the map $\Gamma_r$ admits non-empty, convex and compact values. Eventually, the map $\Gamma_r$ and the set $P$ fulfil the conditions stated in Theorem \ref{theoremsaigal} and we obtain a vector $\bar p\in P$ solving $VI(\Gamma_r,P)$. 
			
			
			Our next goal is to show that the vector $\bar p\in P$ which solves $VI(\Gamma_r,P)$ yields the solution $(\bar p,\bar y)$ for given quasi-variational inequality $QVI(\ref{eqA})$. Thus, it is sufficient to show that inequality (\ref{eq2}) holds for any $z\in M(\bar p)$. For this, we first prove that there exists some $z_{\bar p}\in M(\bar p)\cap B(0,r)$ such that,
			\begin{equation}\label{eq3}
				\langle \bar y^*, \bar y-z_{\bar p}\rangle \geq 0.
			\end{equation}
			In fact, $\bar y\in M(\bar p)\cap \bar B(0,r)$ and in case $\norm{\bar y}<r$, the condition (\ref{eq3}) is satisfied by taking  $z_{\bar p}=\bar y$. On assuming $\norm{\bar y}=r$, we observe $\bar y\in M(\bar p)\setminus \bar B(0,r_{\bar p})$ and there exists some $z_{\bar p}$ satisfying (\ref{eq3}) due to $\mathcal C(\bar p)$. 
			\color{black}Since $\bar y^*,\bar y$ and $z_{\bar p}$ satisfy the conditions (\ref{eq2}) and (\ref{eq3}), we observe from \cite[Lemma 3.1 (ii)]{bianchi},\color{black}
			\begin{equation}\label{eq4}
				\bar y^*\in G(\bar y), \langle \bar y^*, z-\bar y\rangle \geq 0,\enspace \forall\, z\in M(\bar p).
			\end{equation}
			
			On summing inequalities (\ref{eq1}) and (\ref{eq4}), we affirm that $(\bar p, \bar y)$ solves $QVI(\ref{eqA})$:
			$$
			\exists\, \bar y^*\in G(\bar y),\langle g(\bar y), p-\bar p\rangle + \langle \bar y^*, z-\bar y\rangle \geq 0,\enspace \forall\, (p,z)\in P\times M(\bar p).
			$$ \end{proof}
		
		\begin{remark}
			We can obtain \cite[Theorem 3]{ausselradn} as a corollary of Theorem \ref{theorempseudo} by assuming the set $M(P)$ is bounded. Indeed, the following example illustrates that Theorem \ref{theorempseudo} is an actual generalization of \cite[Theorem 3]{ausselradn}:
			\begin{example}
				Consider an affine map $g:\mathbb{R}^2\rightarrow \mathbb{R}^2$ defined as $g(y_1,y_2)=(y_1-1, y_2-2)$. In continuation to Example \ref{eg1}, one may check $G$ is a pseudomonotone map. Furthermore, it fulfils local upper sign-continuity as it is upper semi-continuous map with non-empty convex compact values satisfying $(0,0)\notin G(y)$ for any $y\in \mathbb{R}^2$. Clearly, $M$ is convex valued, lower semi-continuous and closed mapping with $int M(p)\neq \emptyset$ for any $p\in P$. If we fix some $r>1$, then it is easy to verify $M(p)\cap \bar B(0,r)\neq \emptyset$ for any $p\in P$. Hence, $QVI(\ref{eqA})$ admits at least one solution according to Theorem \ref{theorempseudo}. In fact, one can verify $(\bar p,\bar x)=((1,1),(0,0))$ solves the corresponding problem $QVI(\ref{eqA})$.
			\end{example}
		\end{remark}
		\begin{remark}
			We notice that the solution set of the problem $QVI(\ref{eqA})$ is closed under the hypothesis of Theorem \ref{theorempseudo}. In this regard, let us take a sequence $\{(p_n,y_n)\}$ in the solution set of $QVI(\ref{eqA})$ satisfying $\lim_{n\rightarrow \infty}(p_n,y_n)=(\bar p,\bar y)$. Then, we observe that the \cite[Proposition 4.2]{aussel-cotr} together with the continuity of $g$ helps us to prove the vector $(\bar p,\bar y)$ also solves $QVI(\ref{eqA})$. 
		\end{remark}
		\begin{remark}
			The reader may observe that Theorem \ref{theorempseudo} can not be stated in infinite-dimensional space. In the proof of Theorem \ref{theorempseudo}, it occurs that the interior of a compact set $M(p)\cap \bar B(0,r)$ is non-empty, which is only possible if the considered space is finite-dimensional.
		\end{remark}
		
		We observe that the map $G$ is considered as pseudomonotone in Theorem \ref{theorempseudo}, due to which it is not applicable to some useful problems such as quasi-convex optimization \cite{ausselnormal, ausselQVI, aussel-hadjnormal}. This fact motivates us to weaken the assumption of pseudomonotonicity on $G$ and consider it as quasimonotone. 
		In our upcoming existence result, we demonstrate the occurrence of non-trivial solutions for the considered problem $QVI(\ref{eqA})$ by considering the map $G$ is quasi-monotone. 
		
		
		\color{black}
		\begin{theorem}\label{theoremquasi}
			Assume that $g:\mathbb{R}^n\rightarrow \mathbb{R}^m$ is an affine map and $P\subset \mathbb{R}^m$ is non-empty, convex and compact. Suppose that:
			\begin{enumerate}[label=(\roman*),ref=(\roman*)]
				\item  $M:P\rightrightarrows \mathbb{R}^n$ is convex valued, lower semi-continuous and closed mapping where the interior of the set $M(p)$ is non-empty for each $p\in P$;
				\item \label{ii} $G:\mathbb{R}^n\rightrightarrows \mathbb{R}^n$ satisfies the local upper sign-continuity, dual lower semi-continuity and quasimonotonicity over the set $co(M(P))$.
			\end{enumerate}
			Then there exists a solution for $QVI^*(\ref{eqB*})$ if the coercivity condition $\mathcal{C}(p)$ holds for each $p\in P$ and there exists some $r>sup\{r_p\,|\,p\in P\}$ such that $M(p)\cap \bar B(0,r)\neq \emptyset$ for any $p\in P$.
		\end{theorem}
		\begin{proof}
			Define a multi-valued map $M_r:P\rightrightarrows \mathbb{R}^n$ as $M_r(p)=M(p)\cap \bar B(0,r)$. Suppose $\Phi_r^*:P\rightrightarrows \mathbb{R}^n$ is considered as $\Phi_r^*(p)=VI^*(G,M_r(p))$, that is,
			\begin{equation}
				\Phi_r^*(p)=\{\bar y\in M_r(p)|\,\exists\,\bar y^*\in G(\bar y)\setminus\{0\}~\text{satisfying}~\langle \bar y^*,y-\bar y \rangle\geq 0,~\forall y\in M_r(p) \}
			\end{equation}
			and $\Gamma_r^*:P\rightrightarrows \mathbb{R}^m$ is considered as $\Gamma_r^*(p)=g\circ \Phi_r^*(p)$. We first aim to show that $VI(\Gamma_r^*, P)$ admits a solution. We observe that for any $\bar p\in P$ solving $VI(\Gamma_r^*,P)$, there exists some $\bar p^*\in \Gamma_r^*(\bar p)=g\circ \Phi_r^*(\bar p)$ or equivalently $\bar p^*=g(\bar y)$ for some $\bar y\in \Phi_r^*(\bar p)$. Hence, $\bar p\in P$ solving $VI(\Gamma_r^*,P)$ and $\bar y\in M(\bar p)\cap \bar B(0,r)$ solving $VI^*(G,M_r(\bar p))$ satisfy,
			\begin{align}
				\langle g(\bar y), p-\bar p\rangle &\geq 0,\enspace\text{for all}~ p\in P,~\text{and}\label{eq5}\\
				\exists\, \bar y^*\in G(\bar y)\setminus\{0\}, \langle \bar y^*, z-\bar y\rangle &\geq 0,\enspace\text{for all}~ z\in M(\bar p)\cap \bar B(0,r).\label{eq5'}
			\end{align} 
			
			According to the proof of Theorem \ref{theorempseudo}, $M_r$ is lower semi-continuous closed map and the set $M_r(p)$ is convex compact having non-empty interior for each $p\in P$. Note that $G$ and $M_r$ fulfil all the required conditions in \cite[Proposition 4.3]{aussel-cotr}, which proves $\Phi_r^*$ is a closed map. Hence, the map $\Phi_r^*$ meets the upper semi-continuity property due to the fact $\Phi_r^*(P)\subseteq M(P)\cap \bar B(0,r)$ is compact. 	In the view of hypothesis $\ref{ii}$ on the map $G$, we get the following relation for any $p\in P$ due to \cite[Lemma 3.1(v)]{aussel-cotr},
			$$\Phi_r^*(p)=VI^*(G,M_r(p))=MVI(G,M_r(p)).$$
			Hence, we can conclude from Lemma \ref{lemmaNE}, $\Phi_r^*$ is upper semi-continuous map with non-empty convex compact values as the set $M_r(p)$ is compact. 
			
			This in turn implies the map $\Gamma_r^*=g\circ \Phi_r^*:P\rightrightarrows \mathbb{R}^m$ is upper semi-continuous due to continuity of the function $g$. Furthermore, the fact that $g$ is an affine function, leads us to the conclusion that $\Gamma_r^*$ admits non-empty, convex and compact values. Eventually, the map $\Gamma_r^*$ and the set $P$ fulfils the conditions stated in Theorem \ref{theoremsaigal} and we obtain a vector $\bar p\in P$ solving $VI(\Gamma_r^*,P)$.
			
			Finally, it is simple to check that $(\bar p, \bar y)\in P\times M_r(\bar p)$ solving (\ref{eq5}) and (\ref{eq5'}) also solves $QVI^*(\ref{eqB*})$ by following the steps provided in the end of Theorem \ref{theorempseudo}.
		\end{proof}
		
		We can obtain \cite[Theorem 4]{ausselradn} as a corollary of the above-mentioned result Theorem \ref{theoremquasi} by assuming the set $M(P)$ is bounded.
		
		\color{black}	For any $i\in \mathcal{I}=\{1,\cdots,N\}$, consider multi-valued maps $G_i:\mathbb{R}^{n_i}\rightrightarrows \mathbb{R}^{n_i}$ and $M_i:P\rightrightarrows \mathbb{R}^{n_i}$, where $\sum_{i\in \mathcal{I}} n_i=n$. Assume that the maps $G$ and $M$ are replaced by,
		\begin{equation}\label{Gprod}
			G(x)=\prod_{i\in \mathcal{I}} G_i(x_i),M(p)=\prod_{i\in \mathcal{I}} M_i(p)
		\end{equation}
		and $(G(x)\setminus \{0\})$ is replaced by $(G_1(x_1)\setminus \{0\},\cdots, G_N(x_N)\setminus \{0\})$ in $QVI^*(\ref{eqB*})$. It is well known that the map $G$ formed as product of quasi-monotone and locally upper sign-continuous maps $G_i$, need not preserve the similar properties \cite{ausselprod}. By assuming these hypotheses on the component maps $G_i$, we extend Theorem \ref{theoremquasi} to the case where $G$ and $M$ are product maps.		
		\begin{theorem}\label{theoremquasiproduct}
			Assume that $g:\mathbb{R}^n\rightarrow \mathbb{R}^m$ is an affine map and $P\subset \mathbb{R}^m$ is non-empty, convex and compact. Suppose that for any $i\in \mathcal{I}$:
			\begin{enumerate}[label=(\roman*),ref=(\roman*)]
				\item  $M_i:P\rightrightarrows \mathbb{R}^{n_i}$ is convex valued, lower semi-continuous and closed mapping where the interior of the set $M_i(p)$ is non-empty for each $p\in P$;
				\item \label{ii} $G_i:\mathbb{R}^{n_i}\rightrightarrows \mathbb{R}^{n_i}$ satisfies the local upper sign-continuity, dual lower semi-continuity and quasimonotonicity over the set $co(M_i(P))$;
				\item for any $p\in P$,
				\begin{align*}
					\qquad&\exists\, r^i_p>0~\text{such that}~\forall\, y_i\in M_i(p)\setminus \bar B(0,r^i_p), \exists\, z_i\in M_i(p)~\text{with} \\&~\norm{z_i}<\norm {y_i}~\text{satisfying}~ \langle y_i^*,z_i-y_i\rangle\leq 0,~\text{for each}~y_i^*\in G_i(y_i).
				\end{align*} 
			\end{enumerate}
			Let $G$ and $M$ be defined as (\ref{Gprod}) then there exists a solution for $QVI^*(\ref{eqB*})$ if there exists some $r_i>sup\{r^i_p\,|\,p\in P\}$ such that $M_i(p)\cap \bar B(0,r_i)\neq \emptyset$ for any $p\in P$.
		\end{theorem}
		\begin{proof}
			Define multivalued map $\Phi^*_{r_i}:P\rightrightarrows \mathbb{R}^{n_i}$ as, $\Phi^*_{r_i}(p)=VI^*(G_i,M_{r_i}(p))$ where the map $M_{r_i}:P\rightrightarrows \mathbb{R}^{n_i}$ is considered to be $M_{r_i}(p)=M_i(p)\cap \bar B(0,r_i)$. Then, the map $\Phi^*_{r_i}$ is upper semi-continuous with non-empty, convex and compact values as per the proof of Theorem \ref{theoremquasi}. Suppose $\Phi^*_r:P\rightrightarrows \mathbb{R}^n$ is defined as $\Phi^*_r(p)=\prod_{i\in \Lambda} \Phi^*_{r_i}(p)$. By assuming $\Gamma_r^*:P\rightrightarrows \mathbb{R}^m$ as $\Gamma_r^*(p)=g\circ \Phi_r^*(p)$, the existence of solution for $QVI^*(\ref{eqB*})$ can be proved under the given assumptions by following the proof of Theorem \ref{theoremquasi}.
		\end{proof}
		\color{black}
		\subsection{Alternative Existence Result}\label{alternative}
		In the previous subsection, we demonstrated the occurrence of solutions for $QVI(\ref{eqA})$ \color{black}under the generalized monotonicity assumptions, \color{black} through a preliminary existence result for classical variational inequalities. However, it can be observed that $QVI(\ref{eqA})$ is a specific instance of the general framework $QVI(T,K)$ (\ref{QVI}) where $T:D\rightrightarrows \mathbb{R}^m\times \mathbb{R}^n$ is formed as $ T(p,y)= \{(g(y),y^*)|\,y^*\in G(y)\}$ and $K:D\rightrightarrows D$ is defined by $K(p,y)=P\times M(p)$ for the set $D=P\times co(M(P))$. Then, $QVI(\ref{eqA})$ appears as,  
		\begin{align}
			\text{find}~& (\bar p,\bar y)\in K(\bar p,\bar y)~\text{such that}~ \exists\, (g(\bar y),\bar y^*)\in T(\bar p,\bar y)~\text{satisfying}\nonumber\\
			&\langle (g(\bar y),\bar y^*),(p,z)-(\bar p,\bar y)\rangle \geq 0,~\text{for any}~(p,z)\in K(\bar p,\bar y).
		\end{align} 
		
		In this subsection, we present an alternative approach to show the occurrence of solutions for $QVI(\ref{eqA})$ \color{black}without requiring any monotonicity assumption on $G$. \color{black}
		
		\begin{theorem} \label{theoremalt}
			Assume that  $g:\mathbb{R}^n\rightarrow \mathbb{R}^m$ is a continuous function and $P\subset \mathbb{R}^m$ is non-empty, convex and compact. If the following conditions are satisfied:
			\begin{enumerate}[label=(\roman*),ref=(\roman*)]
				\item $M:P\rightrightarrows \mathbb{R}^n$ is a lower semi-continuous and closed map which admits convex values;
				\item $G:\mathbb{R}^n\rightrightarrows\mathbb{R}^n$ is an upper semi-continuous map which admits non-empty, convex and compact values on the set $co(M(P))$.
			\end{enumerate}
			Then $QVI(\ref{eqA})$ consists a solution if the coercivity condition $C(p)$ holds for each $p\in P$ and there exists some $r>sup\{r_p\,|\,p\in P\}$ such that $M(p)\cap \bar B(0,r)\neq \emptyset$ for any $p\in P$.
		\end{theorem}
		\begin{proof}
			\color{black}Define a multi-valued map $M_r:P\rightrightarrows \mathbb{R}^n$ as $M_r(p)=M(p)\cap \bar B(0,r)$. According to the proof of Theorem \ref{theorempseudo}, $M(p)\cap B(0,r)\neq \emptyset$ for any $p\in P$. Hence, the map $M_r(p)$ is lower semi-continuous as per \cite[Corollary 1.2.3]{lsc}. Further, $M_r(p)$ is closed map with non-empty convex compact values in the view of hypothesis (i). 
			
			We observe that the maps $G$ and $M_r$ and the function $g$ fulfil all the conditions in \cite[Theorem 5]{ausselradn}. Hence, there exists $(\bar p,\bar y)\in P\times M_r(\bar p)$ satisfying,
			\begin{align*}
				\langle g(\bar y), p-\bar p\rangle\geq 0,&\enspace\text{for any}~ p\in P,\,\text{and}\\
				\exists\, \bar y^*\in G(\bar y), \langle \bar y^*, y-\bar y\rangle\geq 0,&\enspace\text{for any}~ y\in M(\bar p)\cap \bar B(0,r).
			\end{align*}
			Finally, we can affirm that $(\bar p,\bar y)$ solves $QVI(A)$ by following the steps provided in the end of Theorem \ref{theorempseudo}.  \color{black}
		\end{proof}
		\begin{remark}
			We observe that the results in \cite{ausselcoer} will be applicable to show the existence of solution for $QVI(\ref{eqA})$ only if $T$ defined as $T(p,y)=\{(g(y), y^*)|\, y^*\in G(y)\}$ is assumed to be a pseudomonotone (or quasimonotone) map. However, by following example it is clear that the map $T$ need not satisfy pseudomonotonicity (or quasimonotonicity) property if the map $G$ is considered to be pseudomonotone (or quasimonotone). 
			\begin{example}
				Define $g:\mathbb{R}\rightarrow \mathbb{R}$ as $g(y)=y+1$ and map $G:\mathbb{R}\rightrightarrows \mathbb{R}$ as $G(y)=\{y^2+2\}$. Then, $g$ is an affine function and $G$ is a pseudomonotone map over $\mathbb{R}$. But, the reader can check that $T(p,y)=(y+1,y^2+2)$ fails to be psuedomonotone on the set $[1,5]\times [-1,1]\subseteq \mathbb{R}^2$ by taking the points $(p_1,y_1)=(5,-1)$ and $(p_2,y_2)=(1,1)$.
			\end{example}
		\end{remark}
		\section{Applications} \label{applications}
		\subsection{Application to a Pure Exchange Economy}\label{application1}
		In this subsection, we provide an application of the existence results derived in Section \ref{mainresult} to a pure exchange economic equilibrium problem \color{black} in which consumption sets are not necessarily bounded from below. The existence of equilibrium in this type of exchange economy, where short-selling of securities is prevalent, was first studied by Hart \cite{hart}. Subsequently, different arbitrage notions were introduced to ensure the existence of equilibrium by generalizing the condition of Hart (see \cite{page,won-yannelis} and references therein). \color{black} We obtain a characterization of these economic equilibrium problems in terms of a quasi-variational inequality which is specific instance of the general framework $QVI(\ref{eqA})$. Then, we demonstrate the occurrence of Walrasian equilibrium by applying Theorem \ref{theoremquasiproduct}. 
		
		Consider an economic problem consisting of $N$ consumers and $M$ products/goods. Suppose $\mathcal{I}=\{1,2,\cdots,N\}$ and $\mathcal{L}=\{1,2,\cdots, M\}$ denote the sets of consumers and products, respectively.  Considering $p^l$ to be price associated with the product $l\in \mathcal{L}$, we can assume \color{black} $p=(p^1,\cdots, p^M)\in \mathbb{R}^M$ lies in the set $P=\{p\in \mathbb{R}^M|\norm{p}\leq 1\}$ without loss of generality. Suppose the choice set of any agent $i$ denoted by $X_i\subset \mathbb{R}^M$ is non-empty convex closed (but not necessarily bounded from below). \color{black} Let $y_i=(y_i^1,\cdots, y_i^M)\in X_i$ stands for the consumption plan of a consumer $i\in \mathcal{I}$ and $y=(y_1,\cdots, y_N)\in X$ denotes the consumption of market where $X=\prod_{i\in \mathcal{I}} X_i$. According to the classical theory \cite{page,won-yannelis}, any consumer $i\in \mathcal{I}$ is initially endowed with $e_i\in int X_i$ which he/she aims to exchange for the consumption vector $y_i\in X_i$ having higher utility. 
		The preference of any consumer $i$ with respect to consumption $y_i$ is generally represented by a real valued utility function $u_i:X_i\rightarrow \mathbb{R}$.
		Moreover, for a given price vector $p\in P$, 
		the expenditure $\langle p,y_i\rangle$ of any consumer $i$ can not exceed his initial wealth $\langle p,e_i\rangle$ due to endowments. This leads to a budget constraint,  $$S_i(p)=\{y_i\in X_i\,|\,\langle p,y_i\rangle\leq \langle p,e_i\rangle\}.$$ 
		Clearly, $S_i:P \rightrightarrows X_i$ admits closed and convex values for any $p\in P$. \color{black} Since $X_i$ is not necessarily bounded from below in a pure exchange asset market economy \cite{hart,page,won-yannelis}, $S_i(p)$ need not be compact.\color{black}
		
		Let us recall the definition of an equilibrium for a pure exchange economy \cite{page,won-yannelis},
		\begin{definition}
			For the price vector $\bar p \in P\setminus\{0\}$ and consumption vector $\bar y=(\bar y_i)_{i\in \mathcal{I}}\in \prod_{i\in \mathcal{I}} S_i(\bar p)$, we say $(\bar p, \bar y)$ is Walrasian equilibrium if,
			\begin{align}
				\label{eqdef1} u_i(\bar y_i)=\max_{y_i\in S_i(\bar p)} u_i(y_i),~&\text{for each}~ i\in \mathcal {I},~\text{and}\\
				\label{eqdef2} \sum_{i\in \mathcal{I}} (\bar y_i^l-e_i^l)=\,0,~&\text{for each}~l\in \mathcal{L}.
			\end{align}
		\end{definition}
		\color{black} Following \cite{page}, we define the set of rational allocations as,
		\begin{equation}
			\mathcal{A}=\{y\in X|\,\sum_{i\in \mathcal{I}} (y^l_i-e^l_i)=0,\,\forall\,l\in \mathcal{L}~\text{and}~u_i(y_i)\geq u_i(e_i),\,\forall\,i\in \mathcal{I}\}.
		\end{equation} 
		Suppose $\mathcal{A}_i$ is the projection of set $\mathcal{A}$ on $X_i\subseteq\mathbb{R}^M$.
		\color{black}We assume that the utility functions $u_i:X_i\rightarrow \mathbb{R}$ fulfils the following condition:
		\begin{itemize}
			\item[(A)] $u_i$ is quasiconcave, non-satiable over $\mathcal{A}_i$ and continuous.
		\end{itemize} \color{black}
		In the available literature, the existence of equilibrium through variational inequalities is shown by considering the consumption sets are bounded below and the utility functions are either concave \cite{wets} or semi-strictly quasiconcave \cite{donato}.
		The quasiconcavity of utility functions is more generalized criteria and preferable when compared to concave utility functions (see \cite[Proposition 3.D.3]{collel}). It is well known that in the case of non-smooth concave function, a VI problem involving the subdifferential operator of the assumed function provides a necessary and sufficient optimality condition (see \cite[Theorem 2.1]{milasi}). But in case of non-smooth quasiconcave function, the adjusted normal operator is required in place of subdifferential operator to obtain any solution of optimization problem through corresponding variational inequality (see \cite{aussel-hadjnormal, ausselnormal}).   
		
		\color{black}	Suppose $C\subseteq \mathbb{R}^n$ is convex set and $(C)^{\perp}=\{w^*\in \mathbb{R}^n|\,\langle w^*,w\rangle=\langle w^*,v\rangle,~\forall\, w,v\in C\}$. Let $w,v\in C$ be arbitrary. Then $h:C\rightarrow \mathbb{R}$ is,
		\begin{itemize}
			\item[-] \textit{quasiconcave} if $ h(tw+(1-t)v)\geq \min \{h(w),h(v)\}$ for any $t\in [0,1]$;
			\item[-] \textit{semi-strictly quasiconcave} if it is quasiconcave and $h(tw+(1-t)v)> \min \{h(w),h(v)\}$ for any $t\in (0,1)$ whenever $h(w)\neq h(v)$;
			\item[-] \textit{non-satiated} if for any $y\in C$ there exists $z\in C$ such that $h(z)>h(y)$.
		\end{itemize}
		\color{black}
		It is easy to observe that a function $h$ is quasiconcave over $C$ if and only if the sublevel sets of the corresponding quasiconvex function $-h$, defined as, $S_{-h(w)}=\{v\in C\,|\,(-h)(v)\leq (-h)(w)\}$ is convex for any $w\in C$.  \color{black}Suppose $S^<_{-h(w)}=\{v\in C\,|\,(-h)(v)< (-h)(w)\}$ denotes strict sublevel set. Let $\rho_w=d(w,S^<_{-h(w)})$ for any $w\in C\setminus argmin_{C} (-h)$. 
		The adjusted sublevel set $S^a_{-h}$ is defined as \cite[Definition 2.3]{aussel-hadjnormal},
		\begin{align}\label{normaldef}
			S^a_{-h}(w)=\begin{cases}
				S_{-h(w)}\cap \bar B(S^<_{-h(w)},\rho_w),~\text{if}~w\notin argmin_{C} (-h);\\
				S_{-h (w)},~\text{otherwise}.
			\end{cases}
		\end{align}
		It is easy to observe that $S^<_{-h(w)}\subset S^a_{-h}(w)\subset S_{-h(w)}$ for any $w\in C$.
		\color{black}
		The adjusted normal operator $N^a_{-h}:C\rightrightarrows \mathbb{R}^n$ is taken as normal cone to the adjusted sublevel set (see \cite[Definition 2.5]{aussel-hadjnormal}), that is,  
		$$ N^a_{-h}(w)= \{w^*\in \mathbb{R}^n\,|\,\langle w^*, v-w\rangle \leq 0,~\forall\, v\in S^a_{-h}(w)\}.$$
		These adjusted normal operators are locally upper sign-continuous \color{black} over $C\setminus (argmin_{C} (-h))$ \color{black} (see \cite[Proposition 3.5]{aussel-hadjnormal} or \cite[Corollary 3.2]{homidan}) and quasimonotone (see \cite[Proposition 3.3]{aussel-hadjnormal}) for a given quasiconcave function $h$. 
		\color{black} By employing the arguments given in the proof of \cite[Proposition 4.1]{aussel-hadjnormal}, we derive the following result which will be useful in obtaining a characterization of Walrasian equilibrium problem in terms of $QVI(\ref{eqA})$.
		\begin{proposition}\label{prop_appl}
			For any closed convex set $C\subset \mathbb{R}^n$, let $h:C\rightarrow\mathbb{R}$ be continuous quasiconcave function and $K\subset C$ be non-empty closed convex with $(K)^\perp=\{0\}$. Then any vector $\bar y$ solving $VI(N^a_{-h}\setminus \{0\},K)$ satisfies,
			\begin{equation}
				int S^a_{-h}(\bar y)\cap K=\emptyset.
			\end{equation}
		\end{proposition}
		\begin{proof}
			On contrary, suppose $z_\circ\in int S^a_{-h}(\bar y)\cap K$. As per assumption, 
			\begin{equation}\label{prop_appl_eq}
				\exists\,\bar y^*\in (N^a_{-h}(\bar y)\setminus \{0\})~\text{s.t.}~	\langle \bar y^*, y-\bar y\rangle \geq 0,~\text{for all}~y\in K.
			\end{equation}
			Since $\bar y^*\notin K^{\perp}$, there exists $y_\circ\in K$ such that $\langle \bar y^*,y_\circ-\bar y\rangle \neq 0$ and hence, $\langle \bar y^*,y_\circ-\bar y\rangle > 0$.
			Suppose $z_t=tz_\circ+(1-t)y_\circ$ then $z_t\in K$ for any $t\in [0,1]$ and from (\ref{prop_appl_eq}) we get,
			\begin{equation}\label{prop_appl_eq2}
				\langle \bar y^*, z_t-\bar y\rangle > 0,~\text{for all}~t\in [0,1).
			\end{equation}
			Since, $z_\circ\in int S^a_{-h}(\bar y)$, there exists $t'\in (0,1)$ such that $z_{t'}\in S^a_{-h}(\bar y)$. This implies, $\langle y^*, z_{t'}-\bar y\rangle\leq 0$ for all $y^*\in N^a_{-h}(\bar y)\setminus \{0\}$ and we get contradiction to (\ref{prop_appl_eq2}).
		\end{proof}
			\color{black}
			By following the arguments in \cite[Theorem 2.3.1]{won-yannelis}, \color{black} we will now show that the characterization of above Walrasian equilibrium problem can be obtained in terms of the following quasi-variational inequality $QVI(\ref{appl})$, 
			\begin{align}
				&\text{find}~ (\bar p, \bar y)\in P\times S(\bar p)~\text{such that, there exists}~\bar y^*\in G(\bar y)~\text{fulfilling},\notag\\
				&~~\langle \sum_{i\in \mathcal{I}}(e_i-\bar y_i), p-\bar p\rangle +\langle \bar y^*, y -\bar y\rangle \geq 0,~\text{for all}~\,(p,y)\in P\times M(\bar p). \label{appl}
			\end{align}
			Here the map $M:P\rightrightarrows X$ is defined as $M(p)=\prod_{i\in\mathcal{I}} M_i(p)$ with $M_i(p)=\{y_i\in X_i\,|\,\langle p,y_i-e_i\rangle\leq 1-\norm{p}\}$ and the map $G:X\rightrightarrows \mathbb{R}^{NM}$ is defined as
			$G(y)=\prod_{i\in \mathcal{I}} G_i(y_i)=(N^a_{-u_1}(y_1)\setminus \{0\},\cdots N^a_{-u_N}(y_N)\setminus \{0\})$.
			
			
			\begin{lemma}\label{equivalence}
				Suppose the assumption (A) is fulfilled by $u_i$ for any $i\in \mathcal{I}$. If the vector $(\bar p, \bar y)$ solves the quasi-variational inequality $QVI(\ref{appl})$, then it is Walrasian equilibrium.
			\end{lemma}
			
			\begin{proof}
				\color{black}
				Firstly, we show that (\ref{eqdef2}) holds, that is, $z^l=\sum_{i\in\mathcal{I}}(\bar y^l_i-e^l_i)=0$ for each $l$. On contrary, suppose $z\neq 0$. By taking $y=\bar y$ and $p=\frac{z}{\norm{z}}$ in (\ref{appl}), we get,
				\begin{equation}
					\langle \sum_{i\in \mathcal{I}} (\bar y_i-e_i), \bar p\rangle \geq \langle \sum_{i\in \mathcal{I}} (\bar y_i-e_i), p\rangle=||\sum_{i\in\mathcal{I}}(\bar y_i-e_i)||>0.\label{yi}
				\end{equation} 
				In particular, this shows $\norm{\bar p}= 1$. But, the fact $\bar y_i\in M_i(\bar p)$ for any $i$ along with (\ref{yi}) leads to a contradiction as $0<\sum_{i\in \mathcal{I}} \langle (\bar y_i-e_i), \bar p \rangle \leq N(1-\norm{\bar p})=0$. Hence, $z=0$.
				
				We claim that $\bar y$ satisfies (\ref{eqdef1}). We show that $(M_i(p))^{\perp}=\{0\}$ for any $p\in P$ and $i\in \mathcal{I}$. Since $e_i\in int X_i$, we obtain $\epsilon>0$ such that $e_i\in B(e_i,\epsilon)\subset X_i$. Define $x_i$ as,
				\begin{align}\label{xi}
					x_i=\begin{cases}
						e^l_i+\frac{1}{n},~&\text{if}~p^l<0,\\
						e^l_i,~&\text{if}~p^l=0,\\
						e^l_i-\frac{1}{n},~&\text{otherwise}.
					\end{cases}
				\end{align} 
				Then, $x_i\in B(e_i,\epsilon)\subset X_i$ for sufficiently large $n$ and $\langle p,x_i-e_i\rangle< 1-\norm{p}$. Hence, we obtain $\delta>0$ such that $B(x_i,\delta)\in M_i(p)$. This shows affine hull of $M_i(p)=\mathbb{R}^M$ and $(M_i(p))^{\perp}=\{0\}$ as per \cite{aussel-hadjnormal}. From Proposition \ref{prop_appl}, we observe $(\bar p,\bar y)$ satisfies,
				\begin{equation}\label{lemma_proof}
					int S^a_{-u_i}(\bar y_i)\cap M_i(\bar p)=\emptyset.
				\end{equation}
				We see (\ref{eqdef1}) follows from (\ref{lemma_proof}) if $S^<_{-u_i(\bar y_i)}\subset int S^a_{-u_i}(\bar y_i)$ and $\norm{\bar p}=1$. For this, we show
				\begin{equation}\label{lemma_eq2}
					\langle \bar p,\bar y_i-e_i\rangle =1-\norm{\bar p},~\text{for any}~i\in \mathcal{I}. 
				\end{equation} 
				Since $u_i$ is non-satiable over $\mathcal{A}_i$, there exists $x_i\in X_i$ such that $u_i(x_i)>u_i(\bar y_i)$. It is easy to observe that $x_i\in int S^a_{-u_i}(\bar y_i)$. In fact, we obtain $\delta>0$ such that $B(x_i,\delta)\cap X_i\subseteq S^<_{-u_i(\bar y_i)}\subseteq S^a_{-u_i}(\bar y_i)$ by continuity of $u_i$. 
				Since $\bar{y}_i\in S^a_{-u_i}(\bar y_i)$ (see (\ref{normaldef})), we observe $\bar y_i^t=tx_i+(1-t)\bar y_i\in int S^a_{-u_i}(\bar y_i)$ for all $t\in (0,1]$. If (\ref{lemma_eq2}) is not true, then $\langle \bar p,\bar y_i-e_i\rangle <1-\norm{\bar p}$ and we obtain $t^\circ\in (0,1)$ such that $\bar y_i^{t^\circ}\in M_i(\bar p)$. Hence, $\bar y_i^{t^\circ}\in int S^a_{-u_i}(\bar y_i)\cap M_i(\bar p)$ and we get contradiction to (\ref{lemma_proof}).  
				Finally, the fact $z=0$ along with (\ref{lemma_eq2}) shows $0=\langle \bar p,z\rangle =N(1-\norm{\bar p})$ and $\norm{\bar p}=1$.
			\end{proof}
			
			\color{black}
			Let us recall the following result from \cite{lsc}, which will be helpful in proving the existence of equilibrium:
			\begin{lemma}\cite[Corollary 1.2.4]{lsc} \label{lsc}
				Let $U$ be a metric space, $V$ and $W$ be normed spaces. Suppose that:
				\begin{itemize}
					\item[(a)] $f:U\times W \rightarrow V$ is a continuous function and the function $f(p,.):W\rightarrow V$ is affine for any $p\in U$;
					\item[(b)] the maps $T:U\rightrightarrows V$ and $F:U\rightrightarrows W$ fulfil the lower semi-continuity and admit convex and closed values;
					\item[(c)] for each point $p\in U$, there exists $x\in F(p)$ such that $f(p,x)\in int T(p)$.
				\end{itemize}
				Then the map $R:U\rightrightarrows W$ defined as, $R(p)=\{x\in F(p)|f(p,x)\in T(p)\}$ fulfils the lower semi-continuity property and admits closed convex values.
			\end{lemma}
			
			Now, we prove the existence of Walrasian equilibrium for a given pure exchange economy by applying Lemma \ref{equivalence} and Theorem \ref{theoremquasiproduct}. 
			\color{black}
			\begin{theorem}\label{application}
				The given pure exchange economy admits equilibrium if for each $i\in \mathcal{I}$:
				\begin{itemize}
					\item[(a)] $u_i$
					satisfies the assumption (A) such that the map $N^a_{-u_i}$ fulfils the local upper sign-continuity and dual lower semi-continuity over the set $co(M_i(P))$;
					\item[(b)] there exists $\rho_i>0$ such that 
					for any $p\in P$ and $y_i\in M_i(p)$ with $\norm{y_i}>\rho_i$ there exists $z_i\in M_i(p)$ with $\norm{z_i}\leq  \rho_i$ satisfying:
					\begin{equation}\label{coercive}
						u_i(z_i)>u_i(y_i).
					\end{equation}
				\end{itemize}		
			\end{theorem}\color{black}
			\begin{proof}
				In order to ensure the existence of equilibrium, it is enough to show that $QVI(\ref{appl})$ admits a solution according to Lemma \ref{equivalence}. One can notice that $QVI(\ref{appl})$ is a particular instance of the considered problem $QVI^*(\ref{eqB*})$, where $g:X\rightarrow \mathbb{R}^M$ is formed as 
				$g(y)= \sum_{i\in \mathcal{I}}(e_i-y_i)$, the maps $M:P\rightrightarrows X$ and $G:X\rightrightarrows \mathbb{R}^{NM}$ are, respectively, defined as $M(p)=\prod_{i\in\mathcal{I}} M_i(p)$ and $G(y)=\prod_{i\in \mathcal{I}} G_i(y_i)=\prod_{i\in \mathcal{I}}N^a_{-u_i}(y_i)$.
				
				We derive the occurrence of solution for $QVI(\ref{appl})$ by employing \color{black} Theorem \ref{theoremquasiproduct}. \color{black} In this regard, let us observe that the function $g$ is an affine function and $P\subset \mathbb{R}^M$ is a non-empty convex compact set. We aim to show that for any $i\in \mathcal{I}$, the map $M_i$ is convex valued, lower semi-continuous and closed map where the interior of the set $M_i(p)$ is non-empty for each $p\in P$. The reader may verify that $M_i$ is a closed map by following its definition. To prove $M_i$ is a lower semi-continuous map, let us consider $U=P$, $V=\mathbb{R}$, \color{black} $W=X_i$, $f(p,x_i)=\langle p,x_i-e_i\rangle-(1-\norm{p})$, \color{black} $T(p)=(-\infty,0]$ and \color{black} $F(p)=X_i$ \color{black} in Lemma \ref{lsc}. \color{black}For any $p\in P$, we obtain $x_i\in F(p)$ from (\ref{xi}) such that $f(p,x_i)\in int T(p)$. \color{black} 
				Eventually, we can conclude from Lemma \ref{lsc} that $R\equiv M_i$ is a convex valued lower semi-continuous map. \color{black}Furthermore, for any $p\in P$, we observe that $x_i$ defined in (\ref{xi}) is in $int M_i(p)$ as per the proof of Lemma \ref{equivalence}. \color{black}
				We observe that the assumptions of Theorem \ref{theoremquasiproduct} are satisfied. \color{black}In fact, the map $G_i$ fulfils quasimonotonicity property as per \cite[Proposition 3.3(ii)]{aussel-hadjnormal}. \color{black} We claim that the coercivity condition $(iii)$ in Theorem \ref{theoremquasiproduct} is satisfied for any $i\in \mathcal{I}$. Suppose \color{black}$r^i_p=\rho_i$ \color{black} for any given $p\in P$, then due to our hypothesis for any \color{black}$y_i\in M_i(p)\setminus \bar B(0,\rho_i)$, there exists $z_i\in M_i(p)$ with $\norm {z_i}\leq \rho_i<\norm {y_i}$ \color{black} satisfying $u_i(z_i)>u_i(y_i)$ for each $i\in \mathcal{I}$. This further results into $z_i\in S_{-u_i}^<(y_i)$ for each $i\in \mathcal{I}$. Finally, for each $i\in \mathcal I$ and any $y_i^*\in N_{-u_i}^a(y_i)\subseteq N_{-u_i}^<(y_i)$, we get
				$\langle y^*_i,z_i-y_i\rangle \leq 0$ (see \cite[Corollary 4.4]{aussel-hadjnormal}). 
				
				
				Evidently, the condition $\exists\, r_i>sup \{r^i_p:\,p\in P\}$ such that $M_i(p)\cap \bar B(0,r_i) \neq \emptyset$ for any $p\in P$ is fulfilled \color{black} by choosing $r_i>\max\{\norm{e_i},\rho_i\}$. \color{black} Hence, we ensure the existence of equilibrium by applying Theorem \ref{theoremquasiproduct}.
			\end{proof}
				\color{black}
				\begin{remark} \label{weakex}
					One can derive Theorem \ref{application} by replacing hypothesis $(b)$ with a weaker condition $(b')\,\exists\rho_i>0$ s.t. 
					for any $p\in P$ and $y_i\in M_i(p)$ with $\norm{y_i}>\rho_i$ there exists $z_i\in M_i(p)~\text{with}~\norm{z_i}\leq \rho_i~\text{satisfying}~u_i(z_i)\geq u_i(y_i)~\text{and}~ d(z_i,S^<_{-u_i}(y_i))\leq d(y_i,S^<_{-u_i}(y_i))$ (see (\ref{normaldef})). 
				
			\end{remark}
			Following example fulfils the assumptions of Theorem \ref{application} but \cite[Theorem 4.3]{donato} and \cite[Theorem 4.1]{cotrinatime} are not applicable as $X_i$ are unbounded from below.
			\begin{example} \label{exappl} Suppose $X_1=(-\infty,0]\times [0,\infty), X_2=[0,\infty)\times (-\infty,0]$ and $e_1=(-\frac{1}{2},1),e_2=(1,-\frac{1}{2})$. Let $u_i$ for $i=1,2$ be defined as,
				\begin{align*}	u_i(x_i^1,x_i^2)=\begin{cases}
						-(|x_i^1|+|x_i^2|),&~\text{if}~ |x_i^1|+|x_i^2|\leq \frac{3}{2},\\
						-\frac{3}{2},&~\text{otherwise}.
					\end{cases}
					\end{align*}
					We observe that \cite[Theorem 4.3]{donato} and \cite[Theorem 4.1]{cotrinatime} are not applicable as $X_i$ are unbounded from below and \cite[Theorem 1]{page} is not applicable as $\mathcal{A}=\{(x_1,x_2)\in X_1\times X_2\,|\,x_1+x_2=(\frac{1}{2},\frac{1}{2})\}$ is not compact and $u_i$ is not semi-strictly quasiconcave. However, $u_i$ is continuous, quasiconcave (refer \cite[Example 5.4]{ausselnormal}). We see $u_i$ fulfils Theorem \ref{application} $(a)$ and $(b')$ by taking $\rho_i>\frac{3}{2}$ and $z_i=e_i$. 
					Hence, we obtain an equilibrium $\bar{p}=(-\frac{1}{\sqrt{2}},-\frac{1}{\sqrt{2}})$, $\bar x_1=(0,\frac{1}{2})$ and $\bar x_2=(\frac{1}{2},0)$ as per Lemma \ref{equivalence} and Theorem \ref{application}.  
				\end{example}
				Following example depicts that there may be no equilibrium for an exchange economy with consumption sets unbounded from below, even if $u_i$ satisfy rest of the assumptions in \cite[Theorem 4.3]{donato}.
				\begin{example}\label{unbounded_ex}
					Suppose $\mathcal{I}=\mathcal{L}=\{1,2\}$ and $X_1=\{(x_1^1,x_1^2)\in \mathbb{R}^2\, |\,x_1^2\geq 0\},X_2=\mathbb{R}^2_+$. For agents $i\in \mathcal{I}$, suppose $e_i=(1,1)$ and $u_i:X_i\rightarrow \mathbb{R}$ are defined as $u_1(x^1_1,x^2_1)=x^2_1$ and $u_2(x^1_2,x^2_2)=x^1_2+x^2_2$.
					If $p^1>0$, we observe that agent $1$ will have unbounded income to spend on $x_1^2$ since $X_1$ is unbounded below in $x_1^1$. Further, agent $2$ will have unbounded demand for $x_2^1$ whenever $p^1\leq 0$. Hence, there is no equilibrium.
				\end{example}
				If $u_i$ are continuously differentiable and concave then the following result can be derived by assuming $G_i(x_i)=\{-\nabla u_i(x_i)\}$ in Lemma \ref{equivalence} and Theorem \ref{application}.
				\begin{corollary}
					The given pure exchange economy admits equilibrium if in addition to assumption (A), $u_i$ is continuously differentiable concave function and there exists $\rho_i>0$ such that for any $p\in P$ and $y_i\in M_i(p)$ with $\norm{y_i}>\rho_i$ there exists $z_i\in M_i(p)$ with $\norm{z_i}\leq  \rho_i$ satisfying $u_i(z_i)\geq u_i(y_i)$ for each $i\in \mathcal{I}$.
				\end{corollary}
				\color{black}
				\subsection{Application to Convex GNEP}\label{application2}
				In this section, we will apply Theorem \ref{theoremalt} to ensure the occurrence of equilibrium for a jointly convex generalized Nash games (see \cite{facc,ausselGNEP,cotrinaGNEP} and references therein). Consider a non-cooperative game consisting of $N$-players. Suppose an $i^{th}$ player regulates a strategy variable $x_i\in \mathbb{R}^{n_i}$ where $\sum_{i=1}^{N} n_i=n$. Assume that the vector $x_{-i}\in \mathbb{R}^{n-n_i}$ consists the strategies of players other than $i^{th}$ player and $x=(x_{-i},x_i)\in \mathbb{R}^{n}$ denotes the strategy vector of all the involved players. For a given strategy $x_{-i}$ of rival players, the strategy set of an $i^{th}$ player is confined to, \begin{equation*}\label{Xi}
					X_i(x_{-i})=\{x_i\in \mathbb{R}^{n_i}|\,(x_{-i},x_i)\in X\}
				\end{equation*} 
				where $X$ is a non-empty subset of $\mathbb{R}^{n}$. 
				Suppose the strategy $x_{-i}$ of rivals is known, then the aim of $i^{th}$ player is to optimize his objective function $u_i(x_{-i},.):\mathbb{R}^{n-n_i}\times \mathbb{R}^{n_i}\rightarrow \mathbb{R}$ by selecting a suitable strategy $x_i\in X_i(x_{-i})$, that is, $i^{th}$ player solves following problem,
				$$P_i(x_{-i}): \quad \min_{x_i}~u_i(x_{-i},x_i)~\text{such that}~{x_i\in X_i(x_{-i})}.$$
				For a given strategy $x_{-i}$, suppose $Sol(x_{-i})$ denotes the solution set of the problem $P_i(x_{-i})$. Then $\bar x$ is an equilibrium vector for jointly convex generalized Nash equilibrium problem ($GNEP$) if $\bar x_i\in Sol(\bar x_{-i})$ for all $i\in \{1,2,\cdots,N\}$. \color{black} The normal cone operators $N^<_{u_i}$ are important tool to obtain a solution of $GNEP$ through VI whenever objective functions $u_i$ are quasiconvex with respect to $x_i$ (see \cite{cotrinaGNEP}). For any $(x_{-i},x_i)\in \mathbb{R}^{n}$, suppose $S^<_{u_i(x_{-i},\cdot)}(x_i)=\{y_i\in \mathbb{R}^{n_i}|\,u_i(x_{-i},y_i)<u_i(x_{-i},x_i)\}$. Then, 
				\begin{align*}
					N^<_{u_i}(x_{-i},x_i)=\{x_i^*\in \mathbb{R}^{n_i}|\,\langle x_i^*,y_i-x_i\rangle\leq 0,~\forall\,y_i\in S^<_{u_i(x_{-i},\cdot)}(x_i)\},
				\end{align*}
				$~\text{if}~S^<_{u_i(x_{-i},\cdot)}(x_i)\neq \emptyset$ and $N^<_{u_i}(x_{-i},x_i)=\mathbb{R}^{n_i}$ otherwise.
			\color{black}
		
		We apply Theorem \ref{theoremalt} to derive the existence of equilibrium for the given $GNEP$ by considering the set $X$ is (possibly) unbounded. 
		\begin{theorem}\label{theoremGNEP}
			Assume that $X\subseteq \mathbb{R}^n$ is a non-empty closed convex set. For any $i\in \{1,2,\cdots N\}$, suppose that the function $u_i:\mathbb{R}^{n-n_i}\times \mathbb{R}^{n_i}\rightarrow \mathbb{R}$ is quasiconvex with respect to $x_i$ and continuous for any $x\in \mathbb{R}^{n}$. Then the $GNEP$ consists a solution if there exists some $r'>0$ such that for each $y\in X$ with $\norm y>r'$, there exists $z\in X~\text{with}~\norm{z}<\norm{y}$ satisfying,
			\begin{align}\label{ineqGNEP1}
				u_i(y_{-i},z_i)< u_i(y_{-i},y_i),~\text{for each}~i \in \{1,2,\cdots N\}.
			\end{align}
			
		\end{theorem}
		\begin{proof}
			Suppose $T:\mathbb{R}^n\rightrightarrows \mathbb{R}^n$ is defined as $T(x)=\prod_{i=1}^{N} T_i(x)=\prod_{i=1}^{N}co(N^<_{u_i}(x)\cap S_i(0,1))$.
			To prove the occurrence of equilibrium for given $GNEP$, it is enough to show that $VI(T,X)$ admits a solution (see \cite[Proposition 3.3]{cotrinaGNEP}). We observe that $VI(T,X)$ is a particular instance of $QVI(\ref{eqA})$. In fact, one can assume that $P$ is an arbitrary non-empty, convex and compact subset of $\mathbb{R}^m$, $g:\mathbb{R}^n\rightarrow \mathbb{R}^m$ is zero function, the maps $M:P\rightrightarrows \mathbb{R}^n$ and $G:\mathbb{R}^n\rightrightarrows\mathbb{R}^n$ are, respectively, defined as $M(p)=X$ and $G(x)=T(x)$ in $QVI(\ref{eqA})$.

			We aim to show that the maps $G$ and $M$ fulfil the assumptions of Theorem \ref{theoremalt}. Clearly, the map $T$ defined is upper semi-continuous with non-empty convex and compact values (see \cite[Proposition 3.2]{cotrinaGNEP}) and $M$ is a constant map. We claim that the coercivity condition $\mathcal{C}(p)$ is satisfied for any $p\in P$. \color{black}	Suppose $r_p=r'$ for a given $p\in P$, then due to our hypothesis for any $y\in X \setminus \bar B(0,r')$ there exists $z\in X$ with $\norm z<\norm y$ satisfying $u_i(y_{-i}, z_i)< u_i(y_{-i},y_i)$ 
			for each $i \in \{1,2,\cdots N\}$. Hence, we obtain $z_i\in S^<_{u_i(y_{-i},\cdot)}(y_i)$ and $\langle y^*_i, z_i-y_i\rangle\leq 0~\text{for all}~y^*_i\in N^<_{u_i}(y_{-i},y_i)\supset T_i(y)$. Finally,
			\begin{equation*}
				\langle y^*, y-z\rangle=\sum_{i=1}^{N}	\langle y^*_i, y_i-z_i\rangle\geq 0~\text{for all}~y^*\in T(y).
			\end{equation*}
			\color{black}
			Evidently, there exists $r>r'$ such that $X\cap \bar B(0,r)\neq \emptyset$.
			Finally, the conditions required in Theorem \ref{theoremalt} are satisfied and $VI(T,X)$ admits a solution.
		\end{proof}
		\color{black}
		If $u_i$ are convex with respect to $x_i$ and continuously differentiable then we can derive following result by assuming $T(x)=\prod_{i=1}^{N} \{{\nabla_{i}} u_i (x_{-i},x_i)\}$ in Theorem \ref{theoremGNEP}.
		\begin{corollary}\label{coroGNEP}
			Assume that $X\subseteq \mathbb{R}^n$ is a non-empty closed convex set. For any $i\in \{1,2,\cdots N\}$, suppose that the function $u_i:\mathbb{R}^{n-n_i}\times \mathbb{R}^{n_i}\rightarrow \mathbb{R}$ is convex with respect to $x_i$ and continuously differentiable for any $x\in \mathbb{R}^{n}$. Then the $GNEP$ consists a solution if there exists some $r'>0$ such that for each $y\in X$ with $\norm y>r'$, there exists $z\in X~\text{with}~\norm{z}<\norm{y}$ satisfying,
			\begin{equation*}
				\langle \nabla_i u_i(y),y_i-z_i\rangle\geq 0~\text{for each}~i \in \{1,2,\cdots N\}.
			\end{equation*}			
		\end{corollary}
		\color{black}
		\begin{remark}
			\begin{itemize}
				\item[(i)] We can obtain \cite[Proposition 2.2]{facc} and \cite[Theorem 4.2]{ausselGNEP} as corollary of Theorem \ref{theoremGNEP} by assuming the set $X$ is bounded.
				\color{black}
				\item[(ii)] It can be noted that \cite[Theorem 1]{blum} is not applicable to derive Theorem \ref{theoremGNEP}. Let $\psi$ be Nikaido-Isoda function \cite{Nikaido} defined as $\psi(x,y)=\sum_{i=1}^{N} [u_i(x_{-i},x_i)-u_i(x_{-i},y_i)].$ Suppose $X$ and $u_i$ are,
				$$X=\{(x_1,x_2)\in \mathbb{R}^2|\,x_1\geq 0, 3x_2+4\geq (x_1)^2\}~\text{and}~u_i(x_{-i},x_i)=x_{-i}(x_i)^3,$$ for $i=1,2$. We observe that neither $-\psi$ is monotone nor $-\psi(x,\cdot)$ is convex function as assumed in \cite[Theorem 1 (iii)]{blum}. 
				However, the conditions given in Theorem 4.9 are satisfied and $\bar x=(1,-1)$ forms an equilibrium for this game.
				\item[(iii)] Although the jointly convex $GNEP$ is a particular case of GNEP, we observe Corollary \ref{coroGNEP} is independent of \cite[Theorem 5]{ausselcoer}. 
				Suppose $X$ and $u_i$ are, $$X=\{(x_1,x_2)\in \mathbb{R}^2\,|\,x_1>0,x_2>0,x_1x_2\geq 1\}~\text{and}~ u_i(x_{-i},x_i)=x_i$$ for $i=1,2$. We observe $C_i=$ projection of set $X$ on ${\mathbb{R}}=(0,\infty)$ is not closed as assumed in [11,Theorem 5]. By taking $x_n=(\frac{1}{n},\frac{1}{n})$, we 
				observe $\{x_n\}_{n\in \mathbb{N}}\subset C_1\times C_2$ and 
				there is no such $\rho>0$ which satisfies the assumption $\prod_{i=1}^N K_i(x^n_{-i})\cap \bar B(0,\rho)\neq \emptyset$ for all $n$ as assumed in \cite[Theorem 5]{ausselcoer}. However, $u_i$ and $X$ satisfy the assumptions of Corollary \ref{coroGNEP} and $\bar x=(1,1)$ forms an equilibrium. 
				\item[(iv)] Following example shows that Theorem \ref{theoremGNEP} is independent of \cite[Theorem 2]{castellani}. 
				Suppose $X$ and $C_i$ are same as (iv). Consider $u_i$ as,
				\begin{align*}
					u_i(x_{-i},x_i)=\begin{cases}
						|x_{-i}|+|x_i|,&~\text{if}\,|x_{-i}|+|x_i|\leq 3,\\
						3,&~\text{if}\,|x_{-i}|+|x_i|> 3
					\end{cases}
				\end{align*} 
				for $i=1,2$. By taking $x=(x_{-i},x_i)=(1,3)$ and $z=(x_{-i},z_i)=(1,1)$, we observe that $z_i\in F_i(x)=S^<_{u_i(x_{-i},\cdot)}(x_i)\cap C_i$ but $z_i-x_i\notin R(F_i(x),x_i)$ and \cite[Theorem 2 (i)]{castellani} is not satisfied. However, all the assumptions of Theorem \ref{theoremGNEP} are satisfied and $\bar x=(1,1)$ forms an equilibrium. 
				\color{black}
			\end{itemize}
		\end{remark}
		
		\section*{Acknowledgement(s)}
		The first author acknowledges Science and Engineering Research Board $\big(\rm{MTR/2021/000164}\big)$, India for the financial support. The second author is grateful to the University Grants Commission (UGC), New Delhi, India for the financial assistance provided by them throughout this research work under the registration number: $\big(\rm{1313/(CSIRNETJUNE2019)}\big)$. \color{black} We are grateful to anonymous referees for their invaluable suggestions, which helped us to improve the original manuscript.\color{black}
		
	\end{document}